\documentclass[12pt]{amsart}

\usepackage{amssymb}

\usepackage{enumerate}

\makeatletter
\@namedef{subjclassname@2010}{%
  \textup{2010} Mathematics Subject Classification}
\makeatother

\usepackage{amsmath}
\usepackage{amsthm}

\usepackage{xy}

\usepackage{xypic}

\newtheorem*{main-theorem*}{Theorem}
\newtheorem*{lemma*}{Lemma}

\renewcommand{\mod}{\operatorname{mod}}

\newcommand{\rad}{\operatorname{rad}}

\newcommand{\op}{\operatorname{op}}

\newcommand{\bA}{\mathbb{A}}

\newcommand{\bD}{\mathbb{D}}
\newcommand{\bE}{\mathbb{E}}

\newcommand{\bL}{\mathbb{L}}

\newcommand{\EE}{\mathbb{E}}

\newcommand{\cN}{\mathcal{N}}

\begin{document}

\baselineskip=17pt

\title{Deformed preprojective algebras of Dynkin type $\mathbb{E}_6$}

\author[J. Bia\l kowski]{Jerzy Bia\l kowski}
\address{Faculty of Mathematics and Computer Science,
   Nicolaus Copernicus University,
   Chopina~12/18,
   87-100 Toru\'n,
   Poland}
\email{jb@mat.uni.torun.pl}

\date{}




\begin{abstract}
We prove that every deformed preprojective algebra of Dynkin type $\bE_6$
is isomorphic to the preprojective algebra of Dynkin type $\bE_6$.
\end{abstract}

\maketitle

Throughout this article, $K$ will denote a fixed algebraically 
closed field.
By an algebra we mean an associative finite-dimensional $K$-algebra
with  identity, which we moreover assume to be basic and connected.
For an algebra $A$, we denote by $\mod A$ the category of
finite-dimensional right $A$-modules and by $\Omega_A$ the syzygy
operator which assigns to a module $M$ in $\mod A$ the kernel
of a minimal projective cover $P_A(M) \to M$ of $M$ in $\mod A$.
Then a module $M$ in $\mod A$ is called \emph{periodic} 
if $\Omega_A^n(M) \cong M$ for some $n \geq 1$.
Further, the category of finite-dimensional $A$-$A$-bimodules over
an algebra $A$ is canonically equivalent to the module category
$\mod A^e$ over the enveloping algebra $A^e = A^{\op} \otimes_K A$ of $A$.
Then an algebra $A$ is called a \emph{periodic algebra} 
if $A$ is a periodic module in $\mod A^e$.
It is known that if $A$ is a periodic algebra then is self-injective 
and  
every module $M$ in $\mod A$ without non-zero projective 
direct summands is periodic.
Periodic algebras play currently a prominent r\^ole in  representation
theory of algebras and have attracted much attention (see the survey
article \cite{ESk}).
In particular, it has been proved recently in \cite{Du} that
all self-injective algebras of finite representation type
(different from $K$) are periodic.
We refer also to recent articles \cite{ESk2,ESk3,ESk4}
on the connections of periodic algebras with finite groups
and triangulated surfaces.

In this note we are concerned with the classification 
of deformed preprojective algebras of generalized Dynkin type
$\bA_n (n \geq 2)$, 
$\bD_n (n \geq 4)$, 
$\bE_6$, $\bE_7$, $\bE_8$, 
$\bL_n (n \geq 1)$, 
which are shown in \cite{BES1}
to form the class of all self-injective algebras $A$
for which the third syzygy $\Omega_A^3(S)$ of every non-projective
simple module $S$ in $\mod A$ is isomorphic to its
Nakayama shift $\cN_A(S)$.
We refer to \cite{BES1,ESk},
for results on the importance of these algebras
in the representation theory of self-injective algebras.
The deformed preprojective algebras of generalized Dynkin type 
$\bL_n$ were classified in \cite{BES2},
where it was also shown that they describe 
the stable Auslander algebras of the category of maximal
Cohen-Macaulay modules over simple plane curve singularities
of Dynkin type $\bA_{2n}$.
We refer also to \cite{B} for the classification
of socle deformed preprojective algebras of generalized
Dynkin type.

The main aim of this article is to prove the following theorem 
providing the classification of deformed preprojective
algebras of Dynkin type $\bE_6$.

\begin{main-theorem*}
    Every deformed preprojective algebras of Dynkin type $\bE_6$
    is isomorphic to the preprojective algebras of Dynkin type $\bE_6$.
\end{main-theorem*}

We recall that the 
\emph{preprojective algebra $P(\bE_6)$ of Dynkin type $\bE_6$} 
is given by the quiver
$$
    \begin{array}{c} Q_{\mathbb{E}_{6}}: \\ \end{array}
   \quad \vcenter{
    \xymatrix@C=.8pc@R=2pc{
        && && 0 \ar@<.5ex>[d]^{a_0} \\
        1 \ar@<.5ex>[rr]^{a_1} && 2 \ar@<.5ex>[ll]^{\bar{a}_1} \ar@<.5ex>[rr]^{a_2} &&
        3
         \ar@<.5ex>[ll]^{\bar{a}_2} \ar@<.5ex>[rr]^{a_3} \ar@<.5ex>[u]^{\bar{a}_0} &&
         4 \ar@<.5ex>[ll]^{\bar{a}_3} \ar@<.5ex>[rr]^{a_4} &&
         5 \ar@<.5ex>[ll]^{\bar{a}_4} \\
    }
   }
$$
and the relations
\begin{gather*}
  a_0 \bar{a}_0 = 0, 
  \quad 
  a_1 \bar{a}_1 = 0, 
  \quad 
 \bar{a}_1 a_1 + {a}_2 \bar{a}_2 = 0,
  \quad  
 \bar{a}_0 a_0 + \bar{a}_2 a_2 + a_3 \bar{a}_3 = 0,
 \\
 \bar{a}_3 a_3 + {a}_4 \bar{a}_4 = 0,
  \quad 
  \bar{a}_4 {a}_4 = 0.  
\end{gather*}
We note that $P(\bE_6)$ is not a weakly symmetric algebra, 
and hence not a symmetric algebra.
Further, consider the local commutative algebra
$$
   R(\bE_6) = K \langle x, y \rangle /
   \left( x^2, y^3, (x+y)^{3} \right) 
$$
which is isomorphic to the algebra $e_0 P(\bE_6) e_0$, where $e_0$
is the primitive idempotent in $P(\bE_6)$ associated to the vertex
$0$ of $Q_{\bE_6}$.
An element $f$ from the square $\rad^2 R(\bE_6)$ of the radical
$\rad \ R(\bE_6)$ of $R(\bE_6)$ 
is said to be \textit{admissible}
if $f$ satisfies the following condition
$$\big(x+y + f(x,y)\big)^{3} = 0.$$
Let $f \in {\rm rad}^2R(\bE_6)$ be admissible.
We denote by $P^f(\bE_6)$ 
the algebra  given by the quiver $Q_{\bE_6}$ and the relations
\begin{gather*}
  a_0 \bar{a}_0 = 0, 
  \!\!\quad 
  a_1 \bar{a}_1 = 0, 
  \!\!\quad 
 \bar{a}_1 a_1 + {a}_2 \bar{a}_2 = 0,
  \!\!\quad  
 \bar{a}_3 a_3 + {a}_4 \bar{a}_4 = 0,
  \!\!\quad 
  \bar{a}_4 {a}_4 = 0,  
  \\
 \bar{a}_0 a_0 + \bar{a}_2 a_2 + a_3 \bar{a}_3 
  + f(\bar{a}_0 a_0, \bar{a}_2 a_2) = 0,
  \quad 
  (\bar{a}_0 a_0 + \bar{a}_2 a_2)^3 = 0.
\end{gather*}
Then $P^f(\bE_6)$ is called a 
\emph{deformed preprojective algebra of Dynkin type ${\bE_6}$} 
(see \cite[Section~7]{ESk}).
Observe that $P^f(\bE_6)$ is obtained from $P(\bE_6)$
by deforming the relation at the exceptional vertex $0$ of $Q_{\bE_6}$,
and $P^f(\bE_6) = P(\bE_6)$ if $f = 0$.

The following lemma describes 
the 
admissible elements of ${\rm rad}^2R(\bE_6)$.

\begin{lemma*}
\label{lem:warunek}
An element $f$ from  ${\rm rad}^2R(\bE_6)$ is admissible if and only if
\begin{align*}
 f(x,y) 
  &= 
  \theta_1 
  xy   
  +
  \theta_2 
  yx
  +
  \theta_3 
  yy
  +
  \theta_4 
  xyx
  +
  \theta_5 
  xyy
  +
  \theta_6 
  yxy
  +
  \theta_7 
  xyxy
 \\&
 \quad
  +
  \theta_8 
  yxyy
  +
  \theta_9 
  xyxyy
  ,
\end{align*}
for some $\theta_1, \dots, \theta_9 \in K$
satisfying 
$\theta_2 = 2 \theta_3 - \theta_1$
and
$\theta_6 = 2 \theta_5 - 3 \theta_4 - 3(\theta_3 - \theta_1)^2$.
\end{lemma*}

\begin{proof}
We claim that
$B = \{1_K,x,y,
  xy,   
  yx,
  yy,
  xyx,
  xyy,
  yxy,
  xyxy,
  yxyy,
  xyxyy\}$
is a basis of $R(\EE_6)$ over $K$.
Indeed, it is easy to see, by induction on the degree of elements from $R(\EE_6)$,
that each element $\omega \in R(\EE_6)$
which is a multiplication of elements $x$ and $y$
is a  linear combination (possibly trivial)
of elements from $B$.
In particular we have:
\begin{align*}
  yyx &= yyx - (x+y)^3 = - ( xyx + xyy + yxy), \\
  yxyx &= - xyyx = xyxy, \\
  yyxy &= - (xyxy + yxyy), \\
  yyxyy &= -yyxyx = yxyyx = - yxyxy = xyyxy = - xyxyy.
\end{align*}
Let $f \in {\rm rad}^2R(\EE_6)$.
Then 
\begin{align*}
 f(x,y) 
  &= 
  \theta_1 
  xy   
  +
  \theta_2 
  yx
  +
  \theta_3 
  yy
  +
  \theta_4 
  xyx
  +
  \theta_5 
  xyy
  +
  \theta_6 
  yxy
  +
  \theta_7 
  xyxy
 \\&
 \quad
  +
  \theta_8 
  yxyy
  +
  \theta_9 
  xyxyy
\end{align*}
for some $\theta_1, \dots, \theta_9 \in K$.
Then we have
\begin{align*}
\big(x+y + f(x,y)\big)^3 
  &=
 \big(x+y 
  +
   \theta_1 
  xy   
  +
  \theta_2 
  yx
  +
  \theta_3 
  yy
  +
  \theta_4 
  xyx
  +
  \theta_5 
  xyy
  +
  \theta_6 
  yxy
 \big)^3
 \\
  &=
 \big(x+y 
  +
   \theta_1 
  xy   
  +
  \theta_2 
  yx
  +
  \theta_3 
  yy
 \big)^3
  +
 \big(x+y 
  +
  \theta_4 
  xyx
 \big)^3
 \\&\quad
  +
 \big(x+y 
  +
  \theta_5 
  xyy
 \big)^3
  +
 \big(x+y 
  +
  \theta_6 
  yxy
 \big)^3
 = \dots
 \\
  &=
  (\theta_1+\theta_2-2\theta_3)xyxy
 \\&\quad
  +
  (3\theta_4-2\theta_5+\theta_6+\theta_1^2-\theta_1 \theta_2+\theta_2^2-\theta_3^2)xyxyy
\end{align*}
Hence $f$ is admissible, if and only if
$\theta_1+\theta_2-2\theta_3 = 0$
and 
$3\theta_4-2\theta_5+\theta_6+\theta_1^2-\theta_1 \theta_2+\theta_2^2-\theta_3^2 = 0$.
Moreover, if $\theta_1+\theta_2-2\theta_3 = 0$,
then $\theta_2 = 2 \theta_3 - \theta_1$,
and hence
\begin{align*}
3\theta_4-2\theta_5+\theta_6&+\theta_1^2-\theta_1 \theta_2+\theta_2^2-\theta_3^2 
\\&
= 
3\theta_4-2\theta_5+\theta_6+\theta_1^2-\theta_1 (2 \theta_3 - \theta_1)+(2 \theta_3 - \theta_1)^2-\theta_3^2
\\&
=
3\theta_4-2\theta_5+\theta_6+3\theta_1^2-6\theta_1 \theta_3 +3\theta_3^2
\\&
= \theta_6 - \big( 2 \theta_5 - 3 \theta_4 - 3(\theta_3 - \theta_1)^2\big).
\end{align*}

This ends the proof.
\end{proof}

The remaining part of this article 
is devoted to the proof of Theorem.

\medskip

Let $f$ be an admissible element of ${\rm rad}^2 R(\EE_6)$.
We will show that the algebras
$P(\EE_6)$ and $P^f(\EE_6)$
are isomorphic.
This will be done via a change of generators in $P(\EE_6)$.
%
%
It follows 
from 
Lemma 
that
there exist $\theta_1, \dots, \theta_9 \in K$,
satisfying 
$\theta_2 = 2 \theta_3 - \theta_1$
and
$\theta_6 = 2 \theta_5 - 3 \theta_4 - 3(\theta_3 - \theta_1)^2$
such that
\begin{align*}
 f(x,y) 
  &= 
  \theta_1 
  xy   
  +
  \theta_2 
  yx
  +
  \theta_3 
  yy
  +
  \theta_4 
  xyx
  +
  \theta_5 
  xyy
  +
  \theta_6 
  yxy
  +
  \theta_7 
  xyxy
 \\&
 \quad
  +
  \theta_8 
  yxyy
  +
  \theta_9 
  xyxyy .
\end{align*}
To simplify the notation,
we denote
\begin{align*}
  \alpha &= \theta_4 + (\theta_3 - \theta_1)^2 ,\\
  \beta &= \theta_5 - 2 \theta_4 - 2(\theta_3 - \theta_1)^2 ,\\
  \gamma &=
\theta_7
-8\theta_1\theta_3^2
+7\theta_1^2\theta_3
+2\theta_3\theta_4
-2\theta_1^3
-2\theta_1\theta_4
+3\theta_3^3 ,
\\
  \delta &=
2 \theta_1^4 
-6\theta_1^3\theta_3 
-3\theta_1^2\theta_5 
+4\theta_1^2\theta_4 
+6\theta_1^2\theta_3^2 
+5\theta_1\theta_3\theta_5 
-6\theta_1\theta_3\theta_4 
+\theta_5^2
\\
&\quad
-3\theta_5\theta_4
+2\theta_4^2
-2\theta_3^3\theta_1
-2\theta_3^2\theta_5
+2\theta_3^2\theta_4
+2\theta_1\theta_8
-3\theta_3\theta_8
-\theta_9 ,
\\
  \alpha_1 &= - \alpha  + \theta_1 \theta_3 - \theta_3^2 ,\\
  \beta_1 &= - \beta + \theta_1 \theta_3 - \theta_3^2 ,\\
  \alpha_2 &= - \gamma - \theta_3^2 (\theta_2 - \theta_3) 
              + \theta_1 \beta + \theta_1 \alpha_1 - \theta_3 \alpha_1 ,\\
  \beta_2 &= - \theta_3^2 (\theta_1 - \theta_3) 
              + \theta_3 \alpha + \theta_3 \beta_1 + \theta_3 \beta ,\\
  \alpha_3 &= 
    - \theta_3^2 (\theta_1^2+\theta_2^2+2\theta_4-2\theta_5+\theta_6-\theta_1\theta_2-\theta_2\theta_3) 
\\&\quad
    + \alpha_1 \theta_3^2 (\theta_1 - \theta_3)            
    + \beta \theta_3^2 (\theta_2 - \theta_3)            
    + \theta_3 \theta_8 
\\&\quad
    - \theta_3 \alpha_2 
    - \theta_3 \beta_2 
    - \alpha \alpha_1 
    + \beta \alpha_1 
    - \beta \beta_1 
    - \gamma \theta_3 .
\end{align*}

Now we change generators in $P(\EE_6)$.
We replace 
$a_i$ by $a_i' \in P(\EE_6)$ 
and 
$\bar{a}_i$ by $\bar{a}_i' \in P(\EE_6)$,
for $i \in \{2,3,4\}$, 
defined as follows
\begin{align*}
  a'_2 &= a_2 - \theta_8 a_2 \bar{a}_0 a_0 \bar{a}_2 a_2 \bar{a}_2 a_2 ,
\\
  \bar{a}'_2 &= \bar{a}_2 + 
   \delta
    \bar{a}_0 a_0 \bar{a}_2 a_2 \bar{a}_0 a_0 \bar{a}_2 a_2 \bar{a}_2 ,
\\
  a'_3 &= {a}_3 
   + \theta_1 \bar{a}_0 a_0  {a}_3
   + \theta_3 \bar{a}_2 a_2  {a}_3
   + 
   \alpha
   \bar{a}_0 a_0 \bar{a}_2 a_2  {a}_3
   + 
   \beta
   \bar{a}_2 a_2 \bar{a}_0 a_0  {a}_3
\\
&\quad
   + 
   \gamma
   \bar{a}_0 a_0 \bar{a}_2 a_2 \bar{a}_0 a_0  {a}_3 ,
\\
  \bar{a}'_3 &= \bar{a}_3 + (\theta_3 - \theta_1) \bar{a}_3 \bar{a}_0 a_0
    + (\theta_4 - \theta_5 - \theta_1 \theta_3 + \theta_1^2) \bar{a}_3 \bar{a}_2 a_2 \bar{a}_0 a_0  ,
\\
  a'_4 &= {a}_4 
   + \big((\theta_1 - \theta_3) (2\theta_3 - \theta_1) - \theta_4\big) \bar{a}_3 \bar{a}_0 a_0  {a}_3 {a}_4
\\
&\quad
   + \big(
   3\theta_1\theta_3^2-\theta_3\theta_4-\theta_7-2\theta_1^2\theta_3-\theta_3^3+\theta_1\theta_5
   \big) 
   \bar{a}_3 \bar{a}_2 a_2 \bar{a}_0 a_0 {a}_3 {a}_4 ,
\\
  \bar{a}'_4 &= \bar{a}_4 
   + \big(\theta_4 - (\theta_1 - \theta_3) (2\theta_3 - \theta_1)\big) \bar{a}_4 \bar{a}_3 \bar{a}_0 a_0  {a}_3 .
\end{align*}
and keep all other arrows as they are.
Then
\begin{align*}
  a_2 &= a'_2 + \theta_8 a'_2 \bar{a}_0 a_0 \bar{a}'_2 a'_2 \bar{a}'_2 a'_2 ,
\\
  \bar{a}_2 &= \bar{a}'_2 - 
    \delta
    \bar{a}_0 a_0 \bar{a}'_2 a'_2 \bar{a}_0 a_0 \bar{a}'_2 a'_2 \bar{a}'_2
\\
  a_3 &= a'_3 
   - \theta_1 \bar{a}_0 a_0  a'_3
   - \theta_3 \bar{a}'_2 a'_2  a'_3
   + \alpha_1
      \bar{a}_0 a_0 \bar{a}'_2 a'_2  a'_3
   + \beta_1
      \bar{a}'_2 a_2 \bar{a}_0 a_0  a'_3
\\
&\quad
   + \alpha_2
    \bar{a}_0 a_0 \bar{a}'_2 a'_2 \bar{a}_0 a_0  a'_3
   + \beta_2
    \bar{a}'_2 a'_2 \bar{a}_0 a_0 \bar{a}'_2 a'_2  a'_3
   + \alpha_3
    \bar{a}_0 a_0 \bar{a}'_2 a'_2 \bar{a}_0 a_0 \bar{a}'_2 a'_2  a'_3 .
\\
  \bar{a}_3 &= \bar{a}'_3 - (\theta_3 - \theta_1) \bar{a}'_3 \bar{a}_0 a_0
    - (\theta_4 - \theta_5 - \theta_1 \theta_3 + \theta_1^2) \bar{a}'_3 \bar{a}'_2 a'_2 \bar{a}_0 a_0  
\\ &\quad
    + (\theta_3 - \theta_1)
    (\theta_4 - \theta_5 - \theta_1 \theta_3 + \theta_1^2)
     \bar{a}'_3 \bar{a}_0 a_0 \bar{a}'_2 a'_2 \bar{a}_0 a_0  
\\ &\quad
    +
    (\theta_4 - \theta_5 - \theta_1 \theta_3 + \theta_1^2)^2
     \bar{a}'_3 \bar{a}'_2 a'_2 \bar{a}_0 a_0 \bar{a}'_2 a'_2 \bar{a}_0 a_0  ,
\\
  a_4 &= a'_4 a'_2
   - \big((\theta_1 - \theta_3) (2\theta_3 - \theta_1) - \theta_4\big) \bar{a}'_3 \bar{a}_0 a_0  a'_3 a'_4
\\
&\quad
   - \big(
   3\theta_1\theta_3^2-\theta_3\theta_4-\theta_7-2\theta_1^2\theta_3-\theta_3^3+\theta_1\theta_5
   \big) 
   \bar{a}'_3 \bar{a}'_2 a'_2 \bar{a}_0 a_0 a'_3 a'_4 ,
\\
  \bar{a}_4 &= \bar{a}'_4 
   - \big(\theta_4 - (\theta_1 - \theta_3) (2\theta_3 - \theta_1)\big) 
     \bar{a}'_4 \bar{a}'_3 \bar{a}_0 a_0  a'_3 
\\
&\quad
   + \theta_3 \big(\theta_4 - (\theta_1 - \theta_3) (2\theta_3 - \theta_1)\big) 
     \bar{a}'_4 \bar{a}'_3 \bar{a}_0 a_0 \bar{a}'_2 a'_2 a'_3 .
\end{align*}
Therefore this is an invertible change 
of generators.

We will show now that, with these new generators, 
$P(\EE_6)$ satisfies the relations of $P^f(\EE_6)$.

From the equalities 
$a_2 \bar{a}_2 a_2 \bar{a}_2 = a_2 a_1 \bar{a}_1 \bar{a}_2 = 0$
we obtain
$a_2' \bar{a}'_2 = a_2 \bar{a}_2$,
and hence
\begin{align*}
\bar{a}_1 a_1 + a_2' \bar{a}'_2 
&=  \bar{a}_1 a_1 +  a_2 \bar{a}_2 
= 0 .
\end{align*}
Similarly, from 
$\bar{a}_4 {a}_4 = 0$,
we obtain 
\begin{align*}
\bar{a}'_4 {a}'_4 
&
= 0. 
\end{align*}

Note that from equalities 
$a_4 \bar{a}_4 + \bar{a}_3 a_3 = 0$,
$a_0 \bar{a}_0 = 0$,
$a_2 \bar{a}_2 a_2 \bar{a}_2 = 0$,
$\bar{a}_3 {a}_3 \bar{a}_3 {a}_3 = 0$, and 
${a}_3 \bar{a}_3 + \bar{a}_2 a_2 + \bar{a}_0 a_0 = 0$
we have
\begin{align*}
a_4 \bar{a}_4 \bar{a}_3 \bar{a}_0 a_0  {a}_3
&= - \bar{a}_3 a_3 \bar{a}_3 \bar{a}_0 a_0  {a}_3
= \bar{a}_3 \bar{a}_2 a_2 \bar{a}_0 a_0  {a}_3,\\
\bar{a}_3 \bar{a}_0 a_0  {a}_3 a_4 \bar{a}_4 
&= - \bar{a}_3 \bar{a}_0 a_0  {a}_3 \bar{a}_3 a_3
= \bar{a}_3 \bar{a}_0 a_0 \bar{a}_2 a_2 {a}_3,\\
\bar{a}_3 \bar{a}_2 a_2 \bar{a}_0 a_0  {a}_3 a_4 \bar{a}_4 
&= - \bar{a}_3 \bar{a}_2 a_2 \bar{a}_0 a_0  {a}_3 \bar{a}_3 a_3
= \bar{a}_3 \bar{a}_2 a_2 \bar{a}_0 a_0 \bar{a}_2 a_2 {a}_3, \\
\bar{a}_3 \bar{a}_2 a_2 \bar{a}_0 a_0 \bar{a}_2 a_2 {a}_3 
&= - \bar{a}_3 \bar{a}_2 a_2 \bar{a}_0 a_0  {a}_3 \bar{a}_3 a_3
= \bar{a}_3 \bar{a}_2 a_2 \bar{a}_2 a_2 {a}_3 \bar{a}_3 a_3
\\&
= - \bar{a}_3 \bar{a}_2 a_2 \bar{a}_2 a_2 \bar{a}_0 a_0 a_3
= \bar{a}_3 \bar{a}_2 a_2 {a}_3 \bar{a}_3 \bar{a}_0 a_0 a_3
\\&
= \bar{a}_3 \bar{a}_0 a_0 {a}_3 \bar{a}_3 \bar{a}_0 a_0 a_3
= \bar{a}_3 \bar{a}_0 a_0 \bar{a}_2 a_2 \bar{a}_0 a_0 a_3 .
\end{align*}
Therefore we obtain equalities
\begin{align*}
\bar{a}'_3 a'_3
&= \bar{a}_3 a_3
+ \big((\theta_3 - \theta_1) + \theta_1 - \theta_3 \big) \bar{a}_3 \bar{a}_0 a_0 a_3
\\ &\quad
+ \big(\big(\theta_4 + (\theta_3 - \theta_1)^2\big) + (\theta_3 - \theta_1) \theta_3 \big) \bar{a}_3 \bar{a}_0 a_0 \bar{a}_2 a_2  a_3 
\\ &\quad
+ \big((\theta_4 - \theta_5 - \theta_1 \theta_3 + \theta_1^2) 
\\ &\ \qquad
+ \big(\theta_5 - 2 \theta_4 - 2(\theta_3 - \theta_1)^2\big) \big) \bar{a}_3 \bar{a}_2 a_2 \bar{a}_0 a_0  a_3 
\\ &\quad
+ (\theta_3 - \theta_1) \big(\theta_5 - 2 \theta_4 - 2(\theta_3 - \theta_1)^2\big)
\bar{a}_3 \bar{a}_0 a_0 \bar{a}_2 a_2 \bar{a}_0 a_0  {a}_3
\\ &\quad
+ \theta_3  (\theta_4 - \theta_5 - \theta_1 \theta_3 + \theta_1^2)
\bar{a}_3 \bar{a}_2 a_2 \bar{a}_0 a_0 \bar{a}_2 a_2  {a}_3
\\ &\quad
+ \big(
\theta_7
-8\theta_1\theta_3^2
+7\theta_1^2\theta_3
+2\theta_3\theta_4
-2\theta_1^3
-2\theta_1\theta_4
\\ &\ \qquad
+3\theta_3^3   
\big) \bar{a}_3 \bar{a}_0 a_0 \bar{a}_2 a_2 \bar{a}_0 a_0  {a}_3
\\
\qquad
&= \bar{a}_3 a_3
+ \big(\theta_4 + \theta_1^2 + 2\theta_3^2 - 3 \theta_1 \theta_3 \big) 
( \bar{a}_3 \bar{a}_0 a_0 \bar{a}_2 a_2  a_3 -  \bar{a}_3 \bar{a}_2 a_2 \bar{a}_0 a_0  a_3)
\\ &\quad
+ (
\theta_5 \theta_3 - 2 \theta_4 \theta_3 - 2\theta_3^3 + 4 \theta_3^2 \theta_1 - 2\theta_1^2 \theta_3 
- \theta_5 \theta_1 + 2 \theta_4 \theta_1 + 2\theta_3^2\theta_1 
\\ &\ \qquad
- 4 \theta_1^2 \theta_3 
+ 2\theta_1^3
+ \theta_3 \theta_4 - \theta_3 \theta_5 - \theta_1 \theta_3^2 + \theta_1^2 \theta_3
+
\theta_7
-8\theta_1\theta_3^2
\\ &\ \qquad
+7\theta_1^2\theta_3
+2\theta_3\theta_4
-2\theta_1^3
-2\theta_1\theta_4
+3\theta_3^3   
) 
\bar{a}_3 \bar{a}_0 a_0 \bar{a}_2 a_2 \bar{a}_0 a_0  {a}_3
\\
\qquad
&= \bar{a}_3 a_3
+ \big(\theta_4 + \theta_1^2 + 2\theta_3^2 - 3 \theta_1 \theta_3 \big) 
( \bar{a}_3 \bar{a}_0 a_0 \bar{a}_2 a_2  a_3 -  \bar{a}_3 \bar{a}_2 a_2 \bar{a}_0 a_0  a_3)
\\ &\quad
+ (
\theta_7
+ \theta_1^3
+\theta_3\theta_4
-\theta_1\theta_5
-3\theta_1\theta_3^2
+2\theta_1^2\theta_3
) 
\bar{a}_3 \bar{a}_0 a_0 \bar{a}_2 a_2 \bar{a}_0 a_0  {a}_3
\end{align*}
and
\begin{align*}
a'_4 \bar{a}'_4 
&=  a_4 \bar{a}_4 
+ \big((\theta_1 - \theta_3) (2\theta_3 - \theta_1) - \theta_4\big) \bar{a}_3 \bar{a}_0 a_0  {a}_3 {a}_4 \bar{a}_4
\\ &\quad
+ \big(\theta_4 - (\theta_1 - \theta_3) (2\theta_3 - \theta_1)\big) {a}_4 \bar{a}_4 \bar{a}_3 \bar{a}_0 a_0  {a}_3 
\\ &\quad
+ \big(
3\theta_1\theta_3^2-\theta_3\theta_4-\theta_7-2\theta_1^2\theta_3-\theta_3^3+\theta_1\theta_5
\big) 
\bar{a}_3 \bar{a}_2 a_2 \bar{a}_0 a_0 {a}_3 {a}_4 \bar{a}_4
\\
&=  a_4 \bar{a}_4 
+ (3 \theta_1 \theta_3 - \theta_1^2 - 2\theta_3^2 - \theta_4 \big) 
\bar{a}_3 \bar{a}_0 a_0 \bar{a}_2 a_2  a_3 
\\ &\quad
+ (\theta_4 + \theta_1^2 + 2\theta_3^2 - 3 \theta_1 \theta_3 \big) 
\bar{a}_3 \bar{a}_2 a_2 \bar{a}_0 a_0  a_3
\\ &\quad
+ \big(
3\theta_1\theta_3^2-\theta_3\theta_4-\theta_7-2\theta_1^2\theta_3-\theta_3^3+\theta_1\theta_5
\big) 
\bar{a}_3 \bar{a}_0 a_0 \bar{a}_2 a_2 \bar{a}_0 a_0  {a}_3
\end{align*}
and hence
\begin{align*}
\bar{a}'_3 a'_3 + a'_4 \bar{a}'_4 
&=  a_3 \bar{a}_3 +  a_4 \bar{a}_4 
= 0 . 
\end{align*}

We note that 
we have the equalities
\begin{align*}
\bar{a}_0 a_0 {a}_3 \bar{a}_3 & = - \bar{a}_0 a_0 \bar{a}_2 a_2 , \\
\bar{a}_2 a_2 {a}_3 \bar{a}_3 & = - \bar{a}_2 a_2 \bar{a}_2 a_2 - \bar{a}_2 a_2 \bar{a}_0 a_0 , \\
{a}_3 \bar{a}_3 \bar{a}_0 a_0 & = - \bar{a}_2 a_2 \bar{a}_0 a_0 , \\
\bar{a}_0 a_0 \bar{a}_2 a_2 {a}_3 \bar{a}_3
& = - (\bar{a}_0 a_0 \bar{a}_2 a_2 \bar{a}_0 a_0 + \bar{a}_0 a_0 \bar{a}_2 a_2 \bar{a}_2 a_2) , \\
\bar{a}_2 a_2 \bar{a}_0 a_0 {a}_3 \bar{a}_3 
& = - \bar{a}_2 a_2 \bar{a}_0 a_0 \bar{a}_2 a_2 , \\
\bar{a}_0 a_0 {a}_3 \bar{a}_3 \bar{a}_0 a_0 
& = - \bar{a}_0 a_0 \bar{a}_2 a_2 \bar{a}_0 a_0 ,
\end{align*}
\begin{align*}
\bar{a}_2 a_2 {a}_3 \bar{a}_3 \bar{a}_0 a_0 
& = - \bar{a}_2 a_2 \bar{a}_2 a_2 \bar{a}_0 a_0 
\\ &
= \bar{a}_0 a_0 \bar{a}_2 a_2 \bar{a}_0 a_0 + \bar{a}_0 a_0 \bar{a}_2 a_2 \bar{a}_2 a_2
+ \bar{a}_2 a_2 \bar{a}_0 a_0 \bar{a}_2 a_2 , \\
{a}_3 \bar{a}_3 \bar{a}_2 a_2 \bar{a}_0 a_0 
& = - (\bar{a}_0 a_0 \bar{a}_2 a_2 \bar{a}_0 a_0 + \bar{a}_2 a_2 \bar{a}_2 a_2 \bar{a}_0 a_0) 
\\ &
= \bar{a}_0 a_0 \bar{a}_2 a_2 \bar{a}_2 a_2 + \bar{a}_2 a_2 \bar{a}_0 a_0 \bar{a}_2 a_2 , \\
%
\bar{a}_0 a_0 \bar{a}_2 a_2 \bar{a}_0 a_0 {a}_3 \bar{a}_3 
& = - \bar{a}_0 a_0 \bar{a}_2 a_2 \bar{a}_0 a_0 \bar{a}_2 a_2 , \\
\bar{a}_0 a_0 \bar{a}_2 a_2 {a}_3 \bar{a}_3 \bar{a}_0 a_0 
& = - \bar{a}_0 a_0 \bar{a}_2 a_2 \bar{a}_2 a_2 \bar{a}_0 a_0
= \bar{a}_0 a_0 \bar{a}_2 a_2 \bar{a}_0 a_0 \bar{a}_2 a_2 , \\
\bar{a}_2 a_2 \bar{a}_0 a_0 {a}_3 \bar{a}_3 \bar{a}_0 a_0 
& = - \bar{a}_2 a_2 \bar{a}_0 a_0 \bar{a}_2 a_2 \bar{a}_0 a_0
= - \bar{a}_0 a_0 \bar{a}_2 a_2 \bar{a}_0 a_0 \bar{a}_2 a_2 , \\
\bar{a}_0 a_0 {a}_3 \bar{a}_3 \bar{a}_2 a_2 \bar{a}_0 a_0 
& = - \bar{a}_0 a_0 \bar{a}_2 a_2 \bar{a}_2 a_2 \bar{a}_0 a_0
= \bar{a}_0 a_0 \bar{a}_2 a_2 \bar{a}_0 a_0 \bar{a}_2 a_2 , \\
\bar{a}_2 a_2 {a}_3 \bar{a}_3 \bar{a}_2 a_2 \bar{a}_0 a_0 
& = - \bar{a}_2 a_2 \bar{a}_0 a_0 \bar{a}_2 a_2 \bar{a}_0 a_0
= - \bar{a}_0 a_0 \bar{a}_2 a_2 \bar{a}_0 a_0 \bar{a}_2 a_2 , \\
\bar{a}_0 a_0 \bar{a}_2 a_2 \bar{a}_0 a_0 {a}_3 \bar{a}_3\bar{a}_0 a_0  
& = 0 , \\
\bar{a}_0 a_0 \bar{a}_2 a_2 {a}_3 \bar{a}_3 \bar{a}_2 a_2 \bar{a}_0 a_0 
& = 0 , \\
\bar{a}_2 a_2 \bar{a}_0 a_0 {a}_3 \bar{a}_3 \bar{a}_2 a_2 \bar{a}_0 a_0 
& = - \bar{a}_2 a_2 \bar{a}_0 a_0 \bar{a}_2 a_2 \bar{a}_2 a_2 \bar{a}_0 a_0
= \bar{a}_0 a_0 \bar{a}_2 a_2 \bar{a}_0 a_0 \bar{a}_2 a_2 \bar{a}_2 a_2 .
\end{align*}
Hence we obtain the
equalities
\begin{align*}
a'_3 \bar{a}'_3 
&= 
\Big({a}_3 
+ \theta_1 \bar{a}_0 a_0  {a}_3
+ \theta_3 \bar{a}_2 a_2  {a}_3
+ \big(\theta_4 + (\theta_3 - \theta_1)^2\big)  \bar{a}_0 a_0 \bar{a}_2 a_2  {a}_3
\\
&\qquad
+ \big(\theta_5 - 2 \theta_4 - 2(\theta_3 - \theta_1)^2\big) \bar{a}_2 a_2 \bar{a}_0 a_0  {a}_3
+ \big(
\theta_7
-8\theta_1\theta_3^2
\\ &\qquad \ \ 
+7\theta_1^2\theta_3
+2\theta_3\theta_4
-2\theta_1^3
-2\theta_1\theta_4
+3\theta_3^3   
\big) \bar{a}_0 a_0 \bar{a}_2 a_2 \bar{a}_0 a_0  {a}_3 \Big)
\\
&\quad
\ \cdot
\big(\bar{a}_3 + (\theta_3 - \theta_1) \bar{a}_3 \bar{a}_0 a_0
+ (\theta_4 - \theta_5 - \theta_1 \theta_3 + \theta_1^2) \bar{a}_3 \bar{a}_2 a_2 \bar{a}_0 a_0 \big) 
\\
&= {a}_3 \bar{a}_3 
+ \theta_1 \bar{a}_0 a_0 {a}_3 \bar{a}_3
+ \theta_3 \bar{a}_2 a_2 {a}_3 \bar{a}_3
+ (\theta_3 - \theta_1) {a}_3 \bar{a}_3 \bar{a}_0 a_0
\\& \quad
+ \theta_1 (\theta_3 - \theta_1) \bar{a}_0 a_0 {a}_3 \bar{a}_3 \bar{a}_0 a_0
+ \theta_3 (\theta_3 - \theta_1) \bar{a}_2 a_2 {a}_3 \bar{a}_3 \bar{a}_0 a_0
\\& \quad
+ \big(\theta_4 + (\theta_3 - \theta_1)^2\big)  \bar{a}_0 a_0 \bar{a}_2 a_2 {a}_3 \bar{a}_3
\\& \quad
+ \big(\theta_5 - 2 \theta_4 - 2(\theta_3 - \theta_1)^2\big) \bar{a}_2 a_2 \bar{a}_0 a_0  {a}_3 \bar{a}_3    
\\
& \quad
+ (\theta_4 - \theta_5 - \theta_1 \theta_3 + \theta_1^2) {a}_3 \bar{a}_3 \bar{a}_2 a_2 \bar{a}_0 a_0  
\\& \quad
+ \big(
\theta_7
-8\theta_1\theta_3^2
+7\theta_1^2\theta_3
+2\theta_3\theta_4
-2\theta_1^3
-2\theta_1\theta_4
+3\theta_3^3   
\big) \bar{a}_0 a_0 \bar{a}_2 a_2 \bar{a}_0 a_0  {a}_3  \bar{a}_3
\\& \quad
+ (\theta_3 - \theta_1) \big(\theta_4 + (\theta_3 - \theta_1)^2\big) 
\bar{a}_0 a_0 \bar{a}_2 a_2 {a}_3 \bar{a}_3 \bar{a}_0 a_0
\\& \quad
+ (\theta_3 - \theta_1) \big(\theta_5 - 2 \theta_4 - 2(\theta_3 - \theta_1)^2\big)
\bar{a}_2 a_2 \bar{a}_0 a_0  {a}_3 \bar{a}_3 \bar{a}_0 a_0    
\\& \quad
+ \theta_1 (\theta_4 - \theta_5 - \theta_1 \theta_3 + \theta_1^2)
\bar{a}_0 a_0 {a}_3 \bar{a}_3 \bar{a}_2 a_2 \bar{a}_0 a_0
\\& \quad
+ \theta_3 (\theta_4 - \theta_5 - \theta_1 \theta_3 + \theta_1^2)
\bar{a}_2 a_2 {a}_3 \bar{a}_3 \bar{a}_2 a_2 \bar{a}_0 a_0
\\& \quad
+ \big(\theta_5 - 2 \theta_4 - 2(\theta_3 - \theta_1)^2\big)
(\theta_4 - \theta_5 - \theta_1 \theta_3 + \theta_1^2) 
\bar{a}_2 a_2 \bar{a}_0 a_0 {a}_3 \bar{a}_3 \bar{a}_2 a_2 \bar{a}_0 a_0  
\\
&= {a}_3 \bar{a}_3 
- \theta_1 \bar{a}_0 a_0 \bar{a}_2 a_2
- \big((\theta_3 - \theta_1) + \theta_3\big) \bar{a}_2 a_2 \bar{a}_0 a_0
- \theta_3 \bar{a}_2 a_2 \bar{a}_2 a_2 
\\& \quad
- \big( \theta_1 (\theta_3 - \theta_1) 
- \theta_3 (\theta_3 - \theta_1)
+ \theta_4 + (\theta_3 - \theta_1)^2 \big)
\bar{a}_0 a_0 \bar{a}_2 a_2 \bar{a}_0 a_0
\\& \quad
- \big(\theta_4 + (\theta_3 - \theta_1)^2 
- \theta_3 (\theta_3 - \theta_1)
- (\theta_4 - \theta_5 - \theta_1 \theta_3 + \theta_1^2)
\big)  
\bar{a}_0 a_0 \bar{a}_2 a_2 \bar{a}_2 a_2
\end{align*}
\begin{align*}
\\& \quad
- \big(\theta_5 - 2 \theta_4 - 2(\theta_3 - \theta_1)^2
- \theta_3 (\theta_3 - \theta_1)
\\&\ \ \qquad
- (\theta_4 - \theta_5 - \theta_1 \theta_3 + \theta_1^2)
\big)  
\bar{a}_2 a_2 \bar{a}_0 a_0 \bar{a}_2 a_2  
\\& \quad
- \Big(\big(
\theta_7
-8\theta_1\theta_3^2
+7\theta_1^2\theta_3
+2\theta_3\theta_4
-2\theta_1^3
-2\theta_1\theta_4
+3\theta_3^3   
\big) 
\\
&\ \ \qquad
+ (\theta_3 - \theta_1) \big(\theta_5 - 2 \theta_4 - 2(\theta_3 - \theta_1)^2 
- \big(\theta_4 + (\theta_3 - \theta_1)^2\big)\big)
\\
&\ \ \qquad
+ (\theta_3 - \theta_1) (\theta_4 - \theta_5 - \theta_1 \theta_3 + \theta_1^2)
\Big)
\bar{a}_0 a_0 \bar{a}_2 a_2 \bar{a}_0 a_0 \bar{a}_2 a_2
\\& \quad
+ \big(3 \theta_4 \theta_5 - 2 \theta_4^2 - \theta_5^2 
+ \theta_4 (4\theta_3 \theta_1  - 2\theta_3^2 - 2\theta_1^2 
+ 2 \theta_1 \theta_3 - 2 \theta_1^2)
\\& \ \ \qquad
+ \theta_5 (2\theta_3^2 - 4\theta_3 \theta_1  + 2\theta_1^2 
- \theta_1 \theta_3 + \theta_1^2)
\\& \ \ \qquad
+ 2 \theta_1 (\theta_3 - \theta_1)^3
\big)
\bar{a}_0 a_0 \bar{a}_2 a_2 \bar{a}_0 a_0 \bar{a}_2 a_2 \bar{a}_2 a_2  
\\
&= {a}_3 \bar{a}_3 
- \theta_1 \bar{a}_0 a_0 \bar{a}_2 a_2
- (2\theta_3 - \theta_1) \bar{a}_2 a_2 \bar{a}_0 a_0
- \theta_3 \bar{a}_2 a_2 \bar{a}_2 a_2 
\\& \quad
- \theta_4 \bar{a}_0 a_0 \bar{a}_2 a_2 \bar{a}_0 a_0
- \theta_5 \bar{a}_0 a_0 \bar{a}_2 a_2 \bar{a}_2 a_2
\\& \quad
- \big(2 \theta_5 - 3 \theta_4 - 3(\theta_3 - \theta_1)^2\big)  
\bar{a}_2 a_2 \bar{a}_0 a_0 \bar{a}_2 a_2  
\\& \quad
- \Big(\big(
\theta_7
+ (\theta_3 - \theta_1) \big(2 \theta_4 + 3(\theta_3 - \theta_1)^2 
+ \theta_1 \theta_3 - \theta_1^2\big)
\big) 
\\
&\ \ \qquad
+ (\theta_1 - \theta_3) \big(2 \theta_4 + 3(\theta_3 - \theta_1)^2 
+ \theta_1 \theta_3 - \theta_1^2\big)
\Big)
\bar{a}_0 a_0 \bar{a}_2 a_2 \bar{a}_0 a_0 \bar{a}_2 a_2
\\& \quad
+ (3 \theta_4 \theta_5 - 2 \theta_4^2 - \theta_5^2 
+ 6 \theta_1 \theta_3 \theta_4 
- 4 \theta_1^2 \theta_4 
- 2\theta_3^2 \theta_4  
+ 
2\theta_3^2 \theta_5 
+ 3\theta_1^2 \theta_5 
\\& \ \qquad        
- 5 \theta_1 \theta_3 \theta_5
+ 2 \theta_1^4
- 6 \theta_1^3 \theta_3
+ 6 \theta_1^2 \theta_3^2
- 2 \theta_1 \theta_3^3
)
\bar{a}_0 a_0 \bar{a}_2 a_2 \bar{a}_0 a_0 \bar{a}_2 a_2 \bar{a}_2 a_2  
\\
&= {a}_3 \bar{a}_3 
- \theta_1 \bar{a}_0 a_0 \bar{a}_2 a_2
- \theta_2 \bar{a}_2 a_2 \bar{a}_0 a_0
- \theta_3 \bar{a}_2 a_2 \bar{a}_2 a_2 
- \theta_4 \bar{a}_0 a_0 \bar{a}_2 a_2 \bar{a}_0 a_0
\\& \quad
- \theta_5 \bar{a}_0 a_0 \bar{a}_2 a_2 \bar{a}_2 a_2
- \theta_6  
\bar{a}_2 a_2 \bar{a}_0 a_0 \bar{a}_2 a_2  
- \theta_7
\bar{a}_0 a_0 \bar{a}_2 a_2 \bar{a}_0 a_0 \bar{a}_2 a_2
\\& \quad
+ (3 \theta_4 \theta_5 - 2 \theta_4^2 - \theta_5^2 
+ 6 \theta_1 \theta_3 \theta_4 
- 4 \theta_1^2 \theta_4 
- 2\theta_3^2 \theta_4  
+ 
2\theta_3^2 \theta_5 
+ 3\theta_1^2 \theta_5 
\\& \ \qquad        
- 5 \theta_1 \theta_3 \theta_5
+ 2 \theta_1^4
- 6 \theta_1^3 \theta_3
+ 6 \theta_1^2 \theta_3^2
- 2 \theta_1 \theta_3^3
)
\bar{a}_0 a_0 \bar{a}_2 a_2 \bar{a}_0 a_0 \bar{a}_2 a_2 \bar{a}_2 a_2 . 
\end{align*}

Finally, observe that 
we have also
\begin{align*}
\bar{a}'_2 a'_2 &= \bar{a}_2 a_2 
- \theta_8 \bar{a}_2 a_2 \bar{a}_0 a_0 \bar{a}_2 a_2 \bar{a}_2 a_2 
\\ &\quad
+ 
\big(
2 \theta_1^4 
-6\theta_1^3\theta_3 
-3\theta_1^2\theta_5 
+4\theta_1^2\theta_4 
+6\theta_1^2\theta_3^2 
+5\theta_1\theta_3\theta_5 
-6\theta_1\theta_3\theta_4 
\\
&\quad\quad
+\theta_5^2
-3\theta_5\theta_4
+2\theta_4^2
-2\theta_3^3\theta_1
-2\theta_3^2\theta_5
+2\theta_3^2\theta_4
+2\theta_1\theta_8
-3\theta_3\theta_8 
\\
&\quad\quad
-\theta_9
\big)
\bar{a}_0 a_0 \bar{a}_2 a_2 \bar{a}_0 a_0 \bar{a}_2 a_2 \bar{a}_2 a_2 
\end{align*}
and
\begin{align*}
f(\bar{a}_0 a_0, &\bar{a}'_2 a'_2) 
= f(\bar{a}_0 a_0, \bar{a}_2 a_2)
- \theta_1 \theta_8 \bar{a}_0 a_0 \bar{a}_2 a_2 \bar{a}_0 a_0 \bar{a}_2 a_2 \bar{a}_2 a_2 
\\&\qquad
- \theta_2 \theta_8 \bar{a}_2  a_2 \bar{a}_0 a_0 \bar{a}_2 a_2 \bar{a}_2 a_2 \bar{a}_0 a_0
- \theta_3 \theta_8 \bar{a}_2 a_2 \bar{a}_2 a_2 \bar{a}_0 a_0 \bar{a}_2 a_2 \bar{a}_2 a_2 
\\&
= f(\bar{a}_0 a_0, \bar{a}_2 a_2)
+ (\theta_2 + \theta_3 - \theta_1) \theta_8 \bar{a}_0 a_0 \bar{a}_2 a_2 \bar{a}_0 a_0 \bar{a}_2 a_2 \bar{a}_2 a_2 
\\&
= f(\bar{a}_0 a_0, \bar{a}_2 a_2)
+ \big((2\theta_3 - \theta_1)+ \theta_3 - \theta_1\big) \theta_8 \bar{a}_0 a_0 \bar{a}_2 a_2 \bar{a}_0 a_0 \bar{a}_2 a_2 \bar{a}_2 a_2 
\\&
= f(\bar{a}_0 a_0, \bar{a}_2 a_2)
+ (3\theta_3 - 2\theta_1) \theta_8 \bar{a}_0 a_0 \bar{a}_2 a_2 \bar{a}_0 a_0 \bar{a}_2 a_2 \bar{a}_2 a_2 
.
\end{align*}

Summing up the above equations we obtain
\begin{align*}
\bar{a}_0 a_0 +
\bar{a}'_2 a'_2 +
a'_3 \bar{a}'_3 
+
f(
\bar{a}_0 a_0, \bar{a}'_2 a'_2
) 
=
\bar{a}_0 a_0 + \bar{a}_2 a_2 + a_3 \bar{a}_3 
= 0.
\end{align*}
Hence with these new generators 
$P(\EE_6)$ satisfies the relations of $P^f(\EE_6)$,
and consequently 
the algebras $P(\EE_6)$ and $P^f(\EE_6)$
are isomorphic.


\subsection*{Acknowledgements}
The author gratefully acknowledges support from the research grant
DEC-2011/02/A/ST1/00216 of the National Science Center Poland.

\end{document}